\documentclass[12pt]{amsart}

\usepackage{amsfonts, amstext, amsmath, amsthm, amscd, amssymb, graphicx, epstopdf}
\usepackage{epsfig, graphics, psfrag, enumerate}
\usepackage{color}
\usepackage{verbatim}
\let\oldmarginpar\marginpar
\renewcommand\marginpar[1]{\oldmarginpar[\raggedleft\footnotesize #1]%
{\raggedright\footnotesize #1}}

\newcommand{\labelbp}[1]

\newtheorem{theorem}{Theorem}[section]

\newtheorem{corollary}[theorem]{Corollary}
\newtheorem{proposition}[theorem]{Proposition}
\newtheorem{conjecture}[theorem]{Conjecture}
\newtheorem{question}[theorem]{Question}

\newtheorem{define}[theorem]{Definition}

\theoremstyle{definition}
\newtheorem{remark}[theorem]{Remark}
\newtheorem{example}[theorem]{Example}

\newcommand{\ZZ}{{\mathbb{Z}}}
\newcommand{\NN}{{\mathbb{N}}}

\newcommand{\QQ}{{\mathbb{Q}}}

\newcommand{\bdy}{{\partial}}

\newcommand{\abs}[1]{{\left\vert #1 \right\vert}}

\newcommand\no[1]{}

\newtheorem*{namedtheorem}{\theoremname}
\newcommand{\theoremname}{testing}

\def\be { \begin{equation} }
\def\ee { \end{equation} }

\begin{document}

\title[]{ Remarks on Jones slopes and surfaces of knots}
\thanks{This research is partially supported by NSF grants DMS-1708249 and DMS-2004155} 
\author[]{Efstratia Kalfagianni}
\address{Department of Mathematics, Michigan State University, East Lansing, MI, 48824}
\email{kalfagia@math.msu.edu}

\begin{abstract} We show that the strong slope conjecture implies that the degrees of the colored Jones knot polynomials detect the figure eight knot.
Furthermore, we propose a characterization of alternating knots in terms of the Jones  period and the degree span of the colored Jones polynomial.
\end{abstract}

\thanks {\today}

\maketitle

\section{Introduction} 
The colored Jones polynomial of a knot $K\subset S^3$ is a collection 
of Laurent polynomials 
$ \{ J_K(n):=J_{K}(n, t)\}_{n=1}^\infty $ 
in a variable $t,$  such that  $J_{K}(1)=1$ and $J_{K}(2)$ is the classical Jones polynomial.
In this note  we use the normalization
$$J_{\text{unknot}}(n) = \frac{t^{n/2}- t^{-n/2}}{t^{1/2} -t^{-1/2}}.$$ 

Let $d_+[J_{K}(n)]$ and  $d_-[J_{K}(n)]$ denote  the maximal and minimal degrees of $J_{K}(n)$ in $t$, respectively.
These  degrees are  quadratic quasi-polynomials in $n$. 
The strong slope conjecture asserts that 
they  contain information about essential surfaces in knot exteriors. More specifically, the coefficients of the quadratic terms are {\em{boundary slopes}} of $K$ and the linear terms
encode information about the topology of essential surfaces that realize these boundary slopes.

In \cite{KL} we observed that the strong slope conjecture implies that   $d_+[J_{K}(n)]$ and  $d_-[J_{K}(n)]$ detect the unknot and in \cite{torus} we
show that they detect all the torus knots.
 In this note we show the following.
 
  \begin{theorem} \label{mainint}Let $K$ be knot that satisfies the strong slope conjecture. If the degrees  $d_+[J_{K}(n)]$ and  $d_-[J_{K}(n)]$
  are the same as these of the figure eight knot then $K$ is isotopic to the figure eight knot. 
   \end{theorem}
  
 Theorem \ref{mainint} implies that  the degrees  $d_+[J_{K}(n)]$ and  $d_-[J_{K}(n)]$ detect the figure eight knot within the classes of knots for which
 the strong slope conjecture is known (e.g. the class of adequate knots). The proof of the theorem relies on Gordon's result \cite{Gordon} that gives bounds of the  distance between 
 boundary slopes of punctured tori in irreducible 3--manifolds with toroidal boundary.
 
 We also observe that results on the strong slope conjecture, together with a result of Howie \cite{howie},  suggest the following conjecture that proposes
 a characterization  of alternating knots in terms of their colored Jones polynomial.
 
\begin{conjecture} \label{characterizealtern} Given a knot $K$ let $p_K$  denote the Jones period of $K$.  Then, $K$ is alternating if and only if we have
\begin{equation}
p_K=1 \ \  {\rm and } \ \  2d_+[J_K(n)]- 2d_-[J_K(n)]=c n^2+ (2-c)n-2, 
  \label{eq:alternating}
 \end{equation}
 for some $c\in\ZZ$.
\end{conjecture}

Alternating knots satisfy Equation (\ref{eq:alternating}) with $c=c(K)$, the crossing number of $K$.
Conversely, by \cite{Kaind}  if $K$ is knot that satisfies  Equation (\ref{eq:alternating}) with $c=c(K)$,  then $K$ must be alternating. See Proposition \ref{enough}. Conjecture \ref{characterizealtern}  is seeking to remove the knot diagrammatic reference
to crossing numbers and provide a characterization only in terms of properties of the degree  of $J_K(n)]$.  The conjecture is known to be true for all the knots for which the  strong slope conjecture holds. These include adequate knots, large classes of non-adequate Montesinos knots, graph knots, and knots obtained from these classes by certain satellite operations. See Section 2 for more details.

There are non-alternating knots with Jones period one. For instance, for any adequate knot $K$ we have $p_K=1$ but there exist also families of non-adequate knots
that have this property. On the other hand, alternating knots are the only knots with zero Turaev genus and they form a sub-class of adequate knots. The degree span of the colored Jones polynomial of adequate knots is known to satisfy an analogue of Equation (1) involving the Turaev genus. See  Equation (\ref{eq:adequate}) in Section 4.
We show that this generalized equation, however, does not  characterize adequate knots. See Proposition \ref{counter}.
\smallskip

{\bf Acknowledgement.} The author thanks Christine Lee for useful conversations and in particular for help with Proposition \ref{counter}.
This material is based on research supported by NSF grants DMS-1708249 and DMS-2004155 and by a grant from the Institute for Advanced Study School of Mathematics.
  
\smallskip

\section{Background}

\subsection{Slopes Conjectures}
Garoufalidis  \cite{ga-quasi} proved that
the degrees $ d_+[J_{K}(n)] $ and $ d_-[J_{K}(n)] $ are quadratic { quasi-polynomials:
Given a knot $K$, there is $n_K\in \NN$ such that  for all $n>n_K$ we have
 $$d_+[J_{K}(n)] =  a(n) n^2 + b(n) n  + c(n)\ \ \  {\rm and}  \ \ \   d_-[J_{K}(n)] =  a^{*}(n) n^2 + b^{*}(n) n  + c^{*}(n),$$
 where the coefficients are periodic functions from $\NN $ to $\QQ$. For a sequence $\{x_n\}$, let $\{x_n\}'$ denote the set of its cluster points. 

 \begin{define} {\rm  The  {\em Jones period} of $K$, denoted by $p_K$, 
 is the  least common multiple of the periods
 of these coefficient  functions $a(n), b(n), c(n)$.

 The elements of the sets 
$$js_K:= \left\{ 4n^{-2}d_+[J_K(n)]  \right\}' \quad
 \mbox{and} \quad js^*_K:= \left\{ 4n^{-2}d_-[J_K(n)] \right\}' $$
 are called {\em Jones slopes} of $K$. 
 
 Let  $\ell d_+[J_K(n)]$ and $\ell d_-[J_K(n)]$ denote the linear terms of $d_+[J_K(n)]$ and  $d_-[J_K(n)]$, respectively.
Now we set

$$jx_K:= \left\{ 2n^{-1}\ell d_+[J_K(n)]  \right\}'  \quad
 \mbox{and}\quad  jx^*_K:= \left\{ 2n^{-1}\ell d_-[J_K(n)]  \right\}'.$$}
 \end{define}

\smallskip

Given a knot $K\subset S^3,$ let
  $n(K)$ denote a tubular neighborhood of
$K$ and let $M_K:=\overline{ S^3\setminus n(K)}$ denote the exterior of
$K$. Let $\langle \mu, \lambda \rangle$ be the canonical
meridian--longitude basis of $H_1 (\bdy n(K))$.  
\begin{define} {\rm A properly embedded surface $(S, \bdy S) \subset (M_K,
\bdy n(K))$, is called  {\em essential }if it's $\pi_1$-injective and it is  not a boundary parallel annulus.
An element $\alpha/\beta \in
{\QQ}\cup \{ 1/0\}$, where $\alpha$ and $\beta$ are relatively prime integers,  is called a \emph{boundary slope} of $K$ if there
is an essential surface $(S, \bdy S) \subset (M_K,
\bdy n(K))$, such that each component of  $\bdy S$ represents $\alpha \mu + \beta \lambda \in
H_1 (\bdy n(K))$.} \end{define}

The longitude $\lambda$ of every knot bounds an essential orientable surface in the exterior of $K$.  Thus
 $0=0/1$ is a boundary slope of every knot in $S^3$.
Hatcher showed that every knot $K \subset S^3$
has finitely many boundary slopes \cite{hatcher}. 

For a surface  $(S, \bdy S) \subset (M_K,
\bdy n(K))$ we will use the notation $\abs{\partial S}$ to denote the number of boundary components of $S$.

Garoufalidis conjectured \cite[Conjecture 1.2]{ga-slope}, that the Jones slopes of any knot $K$ are 
 boundary slopes.
The following statement, which  is a refinement of the original conjecture, was 
stated by the author and Tran in \cite[Conjecture 1.6]{Effie-Anh-slope}.

\begin{conjecture}  {\rm ({Strong slope conjecture})} \label{SSC}
\begin{itemize}  
\item Given a Jones slope $a(n)=\alpha/\beta\in js_K$, with $\beta>0$ and $\gcd(\alpha, \beta)=1$, there is an essential surface $S$ in $M_K$ such that each component of $\partial S$ has slope $\alpha/\beta$  and
we have  $\displaystyle{ 2b(n)= { \frac{\chi(S)} {{\abs{\partial S} \beta} } } } \in  jx_K$.

\item  Given  a Jones slope $a^{*}(n)=\alpha^{*}/\beta^{*}\in  js^*_K$,  with $\beta^{*}>0$ and $\gcd(\alpha^{*}, \beta^{*})=1$,
there is an essential surface $S^{*}$ in $M_K$ such that each component of $\partial S^{*}$ has slope $\alpha^{*}/\beta^{*}$ and we have $\displaystyle{2b^{*}(n)=  - \frac{\chi(S^{*})}{{\abs{\partial S^{*}} \beta^{*}}}} \in  jx^*_K$.

\end{itemize}

\end{conjecture}

\begin{remark} Strictly speaking in \cite[Conjecture 1.6]{Effie-Anh-slope} we only required that $\displaystyle{ { \frac{\chi(S)} {{\abs{\partial S} \beta} } } } \in  jx_K$  and 
 $\displaystyle{ - \frac{\chi(S^{*})}{{\abs{\partial S^{*}} \beta^{*}}}} \in  jx^*_K$ without specifying that these values should correspond to points that correspond to the same values of $n$   for  which the slopes  $a(n)$ and $a^{*}(n)$ occur. We don't know if the seemingly 
stronger version statement of  \cite[Conjecture 1.6]{Effie-Anh-slope} is stronger than Conjecture \ref{SSC}. A related point is that, at the moment we don't know if there exist knots for which the sets $ js_K$  or $ js^*_K$
contain more than one point. In all cases for which the Jones slopes are computed, we have exactly one Jones slope in each of $ js_K$  or $ js^*_K$.

\end{remark}
  \smallskip
  
  \subsection{Progress}
Conjecture \ref{SSC} is known for the following families of knots: 
\begin{itemize}
\item  Adequate knots and in particular alternating knots   \cite{FKP, FKP-guts}. 
 \item  Iterated   torus knots and iterated  cables of adequate knots   \cite{TakataMB, Effie-Anh-slope, MoTa}.
\item Graph knots \cite{Takata}.
\item Families of non-alternating   3-tangle pretzel knots \cite{LeeVeen, LeeVeen2} .
\item Families of non-adequate Montesinos knots \cite{LeeGarVeen, LYL, LeeVeen2}.
\item Knots with up to  9 crossings \cite{ga-slope, Howie17, Effie-Anh-slope}.
\item Near-alternating knots \cite{Lee17} constructed by taking Murasugi sums of an alternating diagram with a non-adequate diagram.
\item  Iterated untwisted generalized  Whitehead doubles  of adequate knots  and torus knots \cite{Takata}.
\item Knots obtained by any finite sequence of cabling, connect sums,  and untwisted generalized  Whitehead doubles  of adequate knots and torus knots \cite{TakataMB, Effie-Anh-slope,MoTa}.
\end{itemize}
Under certain conditions Conjecture \ref{SSC} is known to be closed under   cabling operations and  Whitehead doubling operations \cite{Takata, Effie-Anh-slope}.

\smallskip

\section{ Exceptional surgeries and the figure eight knot}
In \cite{torus} we noted  that Conjecture \ref{SSC} implies that the degrees of the colored Jones polynomial distinguish torus knots and in particular the unknot:
  
\begin{theorem} \label{main}Suppose that $K$ is  a knot that satisfies the strong slope conjecture and let $T_{p,q}$ denote the $(p,q)$-torus knot.
If $d_+[J_{K}(n)]=d_+[J_{T_{p,q}}(n)]$ and  $d_-[J_{K}(n)]=d_-[J_{T_{p,q}}(n)]$, for all $n$,  then, up to orientation change,  $K$ is isotopic to  $T_{p,q}$.
\end{theorem}
 
The proof of Theorem \ref{main} begins with the observation that  one of the Jones surfaces for $T_{p,q}$ is an annulus (the cabling annulus). This  implies that  $K$  
also admits a Jones surface of zero Euler characteristic, which in turn implies that $K$ must be a cable knot. The proof of the next theorem is similar in flavor as it begins with the observation that both the Jones surfaces of  the figure eight knot are punctured Klein bottles.

\begin{theorem} \label{mainhere}Suppose that $K$ is  a knot that satisfies the strong slope conjecture and let $F_8$ denote the figure eight knot.
If $d_+[J_{K}(n)]=d_+[J_{F_8}(n)]$ and  $d_-[J_{K}(n)]=d_-[J_{F_8}(n)]$, for all $n$,  then  $K$ is isotopic to  $F_8$.
\end{theorem}
\begin{proof}
 The degrees  $d{\pm}[J_{F_8}(n)]$  are known (see, for example,  \cite{ga-slope, FKP}).
  We have
 $$d_-[J_{K}(n)]=d_-[J_{F_8}(n)]= - n^2+{1\over 2}n+{1\over 2},$$
and 
 $$d_+[J_{K}(n)]=d_+[J_{F_8}(n)]= n^2-{1\over 2}n-{1\over 2}.$$
 Thus we obtain
 
 $$js_K:= \left\{ 4n^{-2}d_+[J_K(n)]  \right\}' =\{ 4\} \quad
 \mbox{and} \quad js^{*}_K:= \left\{ 4n^{-2}d_-[J_K(n)] \right\}' =\left\{-4\right\},$$
and 
$$jx_K:= \left\{ 2n^{-1}\ell d_+[J_K(n)]  \right\}' =\{-1\} \quad
 \mbox{and}\quad  jx^*_K:= \left\{ 2n^{-1}\ell d_-[J_K(n)]  \right\}'=\left\{-1\right\} .$$
Since $K$ satisfies the strong slope conjecture we  have essential surfaces $S,S^{*}$ in the exterior of $K$ such that

\begin{enumerate}
\item the boundary slope of $S$ is $4$ and $\displaystyle{{\chi(S)\over {{\abs{\partial S}} }}}= -1,$
 and
\item  the boundary slope of $S^{*}$ is -4 and $\displaystyle{{\chi(S^{*})\over {{\abs{\partial S^{*}}} }}}= -1.$
\end{enumerate}

This implies that $\chi(S)=-\abs{\partial S}$ and $\chi(S^{*})=-{{\abs{\partial S^{*}}} }.$
Thus  $S$, $S^{*}$ are either punctured tori or punctured Klein bottles.  By passing to the orientable double we can assume that each component of $S,S^{*}$
is a punctured torus.  Thus the knot exterior $M_K$ contains  punctured tori with  boundary slopes $s=4$  and $s^{*}=-4$.  Let  $i(s, s^{*})$
denote  the geometric  intersection of $s, s^{*}$ on $\partial M_K$. In this case we have  $i(s, s^{*})=8$.
By a result of Gordon  \cite[Theorem 1.1]{Gordon} there are only two irreducible 3-manifolds $M$ such that a toroidal component of $\partial M$ contains two boundary slopes  $s, s^{*}$
of punctured tori with $i(s, s^{*})=8$. These are the lowest volume hyperbolic 3-manifolds with one cusp. From these two, the only one that is
a knot complement in $S^3$ is the 
 complement of the figure eight knot. Thus we conclude that $M_K$ is homeomorphic to the complement of $F_8$.  By  the Gordon-Luecke Knot Complement theorem  \cite{GL}, $K$ has to be isotopic to
 $F_8.$
\end{proof}

\smallskip

To continue, we  briefly recall the definition and some notation about adequate knots: Let $D$ be a link diagram, and $x$ be a crossing of $D$.  Associated to
$D$ and $x$ are two link diagrams,  called the \emph{$A$--resolution} and \emph{$B$--resolution} of
the crossing. See Figure \ref{resolutions}.  A {\em{state}} on $D$ is a function $\sigma=\sigma(D)$ that assigns one of these two resolution to each crossing of $D$.
Applying the $A$--resolution (resp. $B$--resolution)  to each crossing leads to  a collection of disjointly embedded circles $s_A(D)$
(resp.  $s_B(D)$).  

\begin{define} {\rm The diagram $D$ is called  \emph{$A$--adequate}  (resp. \emph{$B$--adequate})
if for each crossing of $D$ the two arcs of $s_A(D)$
(resp.  $s_B(D)$) resulting from the resolution of the crossing lie on different circles. A knot diagram $D$
is \emph{adequate} if it is
both $A$-- and $B$--adequate.  Finally, a knot that admits an adequate diagram
is also called \emph{adequate}.}
\end{define}

\begin{figure}
  \includegraphics[scale=1.4]{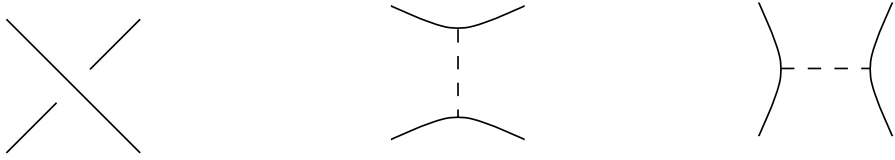}
  
  \caption{From left to right: A crossing, the $A$-resolution and  
 the $B$-resolution.}
    
  \label{resolutions}
\end{figure}
Starting with a state  $\sigma=\sigma(D)$ we construct a {\em{state surface}}  $S_{\sigma}=S_{\sigma}(D)$ as follows:
Each circle of  $\sigma(D)$
bounds a disk on the projection sphere  $S^2\subset S^3$. This
collection of disks can be disjointly embedded in the ball below the
projection sphere. At each crossing of $D$, we connect the pair of
neighboring disks by a half-twisted band to construct a surface
whose boundary is $K$.  For details see \cite{FKP, FKP-guts}.

 The state surfaces corresponding  to $s_A(D)$ and
 $s_B(D)$ are denoted by 
 $S_A(D)$  and   $S_B(D)$, respectfully.
In \cite{ozawa} Ozawa showed that the state surface $S_A(D)$  is essential in the exterior of $K$ if and only if $D$ is an $A$-adequate diagram.
Similarly,  $S_B(D)$  is essential in the exterior of $K$ if and only if $D$ is an $B$-adequate diagram. For a different proof of these results see \cite{FKP-guts}.
Thus, in particular, 
if $D$ is an adequate diagram of a knot $K$ then,
$S_A(D)$  and $S_B(D)$ are essential surfaces in the exterior of $K$. 
\vskip 0.08in

To continue we recall the following well known definition.

\begin{define} {\rm A slope $s$ for a hyperbolic  knot is called  \emph{exceptional} if the 3-manifold obtained by filling $M_K$ along $s$ is not hyperbolic.}
\end{define}
The proof of Theorem \ref{main} shows that both the  Jones slopes of the knot  $F_8$ are exceptional.
Next we will see that  $F_8$ is the only adequate knot that has this property.

 \begin{corollary} \label{co:adequate}Suppose that $K$ is a hyperbolic adequate knot such that
 both the Jones slopes of $K$ are exceptional. Then $K=F_8$.
 \end{corollary}
\begin{proof}

It is known that  the number of negative crossings  $c_-(D)$  of an $A$--adequate knot diagram is a knot invariant.
Similarly, the number of positive crossings $c_+(D)$ of  a  $B$-adequate knot diagram is a knot invariant.
In fact, if $K$ is adequate, then the crossing number of $K$ is realized by the adequate diagram; that is  we have $c(K)=c(D)=c_-(D)+c_+(D)$ \cite{Li}.
Let $v_A(D)$ and  $v_B(D)$ be the numbers of  circles in $s_A(D)$ and $s_B(D)$, respectively.  The boundary slope of $S_A$ is $ -2c_- (D)$ and
$\chi(S_A)= v_A(D)-c(D)$.  The boundary slope of $S_B$ is $ 2c_+ (D)$ and
$\chi(S_B)=v_B(D)-c(D)$. 
By \cite{FKP}, the surfaces $S_A=S_A(D)$ and $S_B=S_B(D)$ satisfy the strong slope conjecture for $K$.

 That is 
we have
$$4\, d_-[J_{K}(n)] =  -2c_- (D) n^2 + 2(c(D) -v_A(D)) n  +2 v_A(D) -2 c_+(D),$$

and

$$
4 \, d_+[J_{K}(n)] = 2c_+ (D)n^2 + 2(v_B(D) - c(D)) n +2 c_-(D)-  2v_B(D).
$$

Thus the distance of the two Jones slopes is $i( 2c_+ (D), \  -2c_- (D))= |2c_+ (D)+2c_- (D)|=2c(D).$ 

Gordon's conjecture, proved by Lackenby and Meyerhoff \cite{LackenbyMa}, states that if $s,s^{*}$ are exceptional boundary slopes for $K$ then
$i(s, \ s^{*})\leq 8$.  Thus in order for $s= 2c_+ (D)$ and $s^{*}= -2c_- (D)$ to be exceptional we must have
$c(D)\leq 4$. Since $K$ is hyperbolic, $K=F_8$.
\end{proof}
\vskip 0.08in

\begin{example} Consider the 3-string pretzel knots $K=P(r, s, t)$ such that  $r<0 < s, t$ and $-2r<s,t$. It has exactly  two Jones slopes with distance $2(s+t)$ (see Proposition \ref{counter} below).
Since by assumption $s,t>2$, we cannot have $2(s+t)\leq 8.$ Thus not both of the Jones slopes can be exceptional.

As another example, let us mention the knot  pretzel knot $P(-2, 3, 7)$, which is  known to have seven exceptional slopes.
The Jones slopes of $P(-2, 3, 7)$ are $\{ \frac{37}{2}, 0\}$ and from these only $\frac{37}{2}$ is exceptional.
\end{example}

\begin{question} Are there hyperbolic knots, other than the figure eight, that have more than one exceptional Jones slopes?
\end{question}

\section{Characteristic Jones surfaces and alternating knots}
We begin by recalling from \cite{KL} that in all the cases where Conjecture \ref{SSC} is proved,  for each Jones slope we can find a Jones surface where 
the {\em{number of sheets}}
 $b|\partial S|$   divides the Jones period $p_K$. This observation led  to the following definition \cite[Definition 3.2]{KL}.
\begin{define}\label{characteristic} {\rm
We call a Jones surface $S$ of a knot $K$ {\emph {characteristic}} if the number of sheets of $S$ divides the Jones period of $K$. }
\end{define}

\begin{example} \label{rem:adequate} An adequate knot (and thus in particular an alternating knot) has Jones period $p_K=1$, two  Jones  slopes and two corresponding
Jones surfaces each with a single boundary component \cite{survey, FKP}.
Note, that  the characteristic Jones surfaces are spanning surfaces that are often non-orientable. In these cases the orientable double cover is also a Jones surface
but it is no longer characteristic since it has two boundary components. 
\end{example}

\begin{question} \label{lessthanp}
Is it true that for every Jones slope of a knot $K$ we can find a characteristic Jones surface?
\end{question}

If $K$ is an alternating  knot then  we have 
$$p_K=1 \ \  {\rm and} \ \   2d_+[J_K(n)]- 2d_-[J_K(n)]=c n^2+ (2-c)n-2,$$

where $c=c(K)$ is the crossing number of $K$.  Thus, one direction Conjecture \ref{characterizealtern} is known.
Furthermore,  an alternating knot  $K$ satisfies the strong slope conjecture and  every Jones slope is realized by a characteristic  Jones surface.  This follows, for example, from the discussion in the proof of 
Corollary  \ref{co:adequate}. The state surfaces $S_A(D), S_B(D)$ corresponding to any reduced alternating diagram $D=D(K)$ are in fact the checkerboard surfaces of $D$.

We have the following converse:

\begin{theorem} \label{alternating} Suppose that $K$ is a knot that satisfies the strong slope conjecture and such that every Jones slope is realized by a characteristic  Jones surface.
Suppose, moreover,  that we have 

$$ p_K=1 \ \  {\rm and} \ \ 2d_+[J_K(n)]- 2d_-[J_K(n)]= cn^2+(2-c)n-2,$$  for some $c\in \ZZ$.
Then, $K$ is alternating and $c$ is the crossing number of $K$.
\end{theorem}

\begin{proof} 
Since we have $p_K=1$,  for each of $d_{\pm }[J_K(n)]$ we have exactly one Jones slope. That is, we have
$js_K= \{ s\}$ and
$js^*_K=\{ s^{*}\}$.
Furthermore, since  
knots of period one have integer  Jones slopes (\cite[Lemma 1.11]{ga-slope}, \cite[Propositrion 3.1]{KL}), both of  $s$, and $s^{*}$ are integers.

Since we assumed that  each Jones slope of $K$  is realized by a characteristic  Jones surface, we conclude that  we can take the Jones surfaces, say $S, S^{*}$, corresponding to
$s, s^{*}$, respectively, to be spanning surfaces of $K$.  

Finally, since we assumed that $2d_+[J_K(n)]- 2d_-[J_K(n)]= cn^2+(2-c)n-2$, for some $c\in \ZZ$,
we conclude that  
$i(\partial S, \partial S^{*})=s-s^{*}=2c,$
where $i(\partial S, \partial S^{*})$ denotes the geometric intersection of the curves $\partial S, \partial S^{*}$ on $\partial M_K$, and
that $\chi(S)+ \chi(S^{*})=2-c$.  Thus in particular, we have

$$\chi(S)+ \chi(S^{*})+ {1\over 2} i(\partial S, \partial S^{*})=2.$$ 

By Howie's result  \cite[Theorem 2]{howie} it follows that $K$ is alternating  and in fact
$S, S^{*}$ are the checkerboard surfaces corresponding to an alternating diagram of $K$. But then $c=c(K)$  by the discussion before the statement of the theorem. \end{proof}
 
As a corollary of Theorem \ref{alternating} we have the following.
\begin{corollary} Suppose that $K$ is an adequate knot. Then $K$ is alternating if and only if we have
$$ 2d_+[J_K(n)]- 2d_-[J_K(n)]= cn^2+(2-c)n-2,  $$
 for some $c\in\ZZ$.
\end{corollary}
\begin{proof}  Conjecture \ref{SSC} has been proved for adequate knots \cite{Effie-Anh-slope}.  Furthermore, as  discussed earlier, adequate knots 
have period one and for every Jones slope we can find a characteristic Jones surface.
Thus, the conclusion follows from Theorem \ref{alternating}.

\end{proof}

\begin{remark}{\rm Theorem  \ref{alternating}  shows that the strong slope conjecture together with a positive answer to Question \ref{lessthanp} implies Conjecture
\ref{characterizealtern} stated in the  Introduction. We also note, that if Conjecture \ref{characterizealtern}  is true then  the degrees $d_{\pm}[J_K(n)]$ would detect 
an alternating knot as long as they detect it among alternating knots with the same crossing number.  That is, if we had a knot $K$ such that  $d_{\pm}[J_K(n)]= d_{\pm}[J_{K'}(n)]$,
where $K'$ is alternating, then  by Theorem \ref{alternating} we would conclude that $K$ is also alternating with $c(K)=c(K')$. So if  $d_{\pm}[J_K(n)]$ distinguishes $K'$ among alternating knots of the same crossing number,
it will detect it among all knots.  Given a prime reduced alternating diagram $D=D(K)$, the degrees $d_{\pm}[J_K(n)]$ are completely determined by the quantities 
$c_- (D), c_+ (D), v_B(D), v_A(D)$, introduced in the proof of Corollary \ref{co:adequate}. Thus,
since $v_B(D), v_A(D)$ are also the numbers of the checkerboard regions of $D$,
 given two reduced alternating diagrams of the same crossing number one can decide whether they are distinguished by the degree of their colored Jones polynomial by a direct diagrammatic inspection.
We illustrate these points with a few examples. We already know that Conjecture \ref{SSC}  implies that  $d_{\pm}[J_K(n)]$, detect the trefoil and figure eight knots.
\begin{itemize} 
\item  If Conjecture \ref{characterizealtern}  is true then  the degrees $d_{\pm}[J_K(n)],$  would detect
the $5_2$ knot: For, suppose that for  a knot $K$ the degrees  $d_{\pm}[J_K(n)]$ are the same as these of  the knot $5_2$.  Then by Theorem \ref{alternating}
$K$ is alternating,  and we have  $c(K)=5$. Thus $K=T_{2, 5}$ or $K=5_2$. Since $T_{2, 5}$ is distinguished from $5_2$ by the degrees of the colored Jones polynomial we conclude that
$K=5_2$. 

\item Now we discuss alternating knots with crossing number six: The degrees $d_{\pm}[J_K(n)]$ distinguish
the knots $6_1, 6_2, 6_3$ from each other. More specifically, the  quantities $c_- (D), c_+ (D)$ (a.k.a. the Jones slopes) distinguish  $6_1$ from $6_2$ and $6_3,$ while the quantities $ v_B(D), v_A(D)$
distinguish $6_2$ and $6_3$ from each other. Hence, if Conjecture \ref{characterizealtern}  is true then,  the degrees $d_{\pm}[J_K(n)]$ would detect  any of $6_1, 6_2, 6_3.$
\end{itemize}}
\end{remark}

The next proposition shows that in order to prove  Conjecture
\ref{characterizealtern}, it is enough to show that if $K$ is a knot that satisfies equation \ref{eq:alternating}, then we must have $c=c(K)$.
\begin{proposition} \label{enough} If  $K$ is a knot such that
$$
2d_+[J_K(n)]- 2d_-[J_K(n)]=c(K) n^2+ (2-c(K))n-2, 
$$
then $K$ is alternating.
\end{proposition} 
\begin{proof} Let $D$ be a knot diagram of $K$ that realizes $c(K)$ and let $g_T(D)$ denote the Turaev genus of $D$ \cite{TG}.
Then we have
$$
2d_+[J_K(n)]- 2d_-[J_K(n)]\leq c(K) n^2+ (2-c(K)-2g_T(D))n+2g_T(K)-2,
$$
for all $n\in \NN$. See for example \cite{Kaind}. Thus we must have $2-c(K)\leq 2-c(K)-2g_T(D)$, which implies $g_T(D)=0$. 
Now by \cite[Corollary 4.6]{TG}, $D$ must be an alternating diagram.
\end{proof}

 As mentioned above, alternating knots are the only knots that have Turaev genus zero \cite[Corollary 4.6]{TG}.  Thus the degree span condition in  the statement of  Conjecture \ref{characterizealtern} can be reformulated to say
\begin{equation}
2d_+[J_K(n)]- 2d_-[J_K(n)]=c n^2+ (2-2g_T(K)-c)n+2g_T(K)-2, 
\label{eq:adequate}
\end{equation}
where $g_T(K)$ denotes the Turaev genus of $K$ and $c$ is an integer.  By \cite{Kaind} adequate  knots satisfy condition (2) above, and they have Jones period equal to one.
One can ask whether  these conditions characterize adequate knots. The following proposition, shown to me by Christine Lee, shows that this is not the case.

\begin{proposition} \label{counter} Consider any  3-string pretzel knot $K=P(r, s, t)$  with  $r<0 < s, t$ and $-2r<s,t.$ Then we have $ p_K=1$ and

$$2d_+[J_K(n)]- 2d_-[J_K(n)]=c n^2+ (2-2g_T(K)-c)n+2g_T(K)-2,$$
but $K$ is non-adequate.
\end{proposition}
\begin{proof}

 The standard 3-string pretzel diagram $D$  of $K=P(r, s, t)$ has $s+t-r$ crossings
and by \cite{LT} this is the crossing number of $K$. That is $c(K)=c(D)=s+t-r$.
The diagram $D$ is also $B$-adequate with  $c_+(D)=c(D)=s+t-r$ and $v_B(D)=-r+1.$ Thus we have

$$2\, d_+[J_{K}(n)] = (s+t-r) n^2 + (-s-t+1)n +(r-1).$$
On the other hand,  Lee \cite{Lee17} shows  that 
\begin{itemize}
\item  the Jones slope coming from $d_+[J_{K}(n)]$
is equal to $2c_-(D)-2r=-2r$;
\item  it is realized by a Jones surface
that is actually the state surface $S_{\sigma}$ corresponding to the state $\sigma$ that assigns the $-r$ crossings the $B$-resolution and the $s+t$ crossings the $A$-resolution.
\end{itemize}
Note that the  hypothesis $-2r<s,t$ is needed for these claims.

The number of state circles for  above state $\sigma$ is given by $v_{\sigma}(D)=-r-1+s-1+t-1+2=-r+s+t-1$. We have  $$-\chi(S_{\sigma})= -(v_{\sigma}(D)-c(D))=-(-r+s+t-1-s-t+r)=1,$$ and
$$2\, d_-[J_{K}(n)] = - rn^2 + n +(r-1).$$
It follows that

$$2\, d_+[J_{K}(n)]-2\, d_-[J_{K}(n)]= (s+t)n^2 -(s+t) n.$$
The Turaev genus of non-alternating  3-string pretzel knots is known to be one and hence $2-2g_T(K)=0$.
With this observation we see that the last equation can be written in the form of Equation (\ref{eq:adequate}) where $c=s+t\in \ZZ$.
Finally  since $s+t< s+t-r=c(K)$, the knot $K$ is not adequate.
\end{proof}

\smallskip

\begin{remark} In \cite{Kaind} we show that if in Equation (\ref{eq:adequate}) we require that the constant $c$ is actually the crossing number of $K$, then $K$ must be adequate.
Proposition \ref{counter} and its proof show that the condition $c=c(K)$ is necessary.
\end{remark}

\begin{remark} The proof of Proposition \ref{counter} shows, in particular, that there are non-adequate knots $K$  that admit spanning surfaces $S,S^{*}$ such that
$$\chi(S)+\chi(S^{*})+{1\over 2} i(\partial S, \partial S^{*})=2-2g_T(K).$$
This should be compared with the main result of \cite{howie} that states that for $g_T(K)=0$ this equation characterizes alternating knots and with  \cite[Problem 1.3]{Kaind}.
\end{remark}


\bibliographystyle{plain} \bibliography{biblio}
\end{document}